\documentclass[10pt,oneside,reqno]{amsart}
\usepackage{amsmath,amssymb,amsthm}
\usepackage{mathtools}
\usepackage{mathrsfs}

\usepackage[includeheadfoot]{geometry}
\usepackage[colorlinks=true,linkcolor=blue,citecolor=blue,pdfpagelabels=false]{hyperref}
\newtheorem{thm}{Theorem}[section]
\newtheorem{lem}[thm]{Lemma}

\newtheorem{cor}[thm]{Corollary}

\newcommand{\thmref}[1]{Theorem~\ref{#1}}
\newcommand{\lemref}[1]{Lemma~\ref{#1}}

\newcommand{\propref}[1]{Proposition~\ref{#1}}
\newcommand{\corref}[1]{Corollary~\ref{#1}}

\newtheorem{rmk}[thm]{Remark}

\newenvironment{acknowledgements}{\bigskip\textbf{Acknowledgements.}}{}

\numberwithin{equation}{section}

\theoremstyle{plain}
\newtheorem{theorem}{Theorem}[section]
\newtheorem{lemma}[theorem]{Lemma}

\newtheorem{proposition}[theorem]{Proposition}

\theoremstyle{definition}

\newtheorem{remark}[theorem]{Remark}

\newtheorem{hyp}[theorem]{Hypothesis}

\newcommand{\Q}{{\mathbb Q}}

\newcommand{\Z}{{\mathbb Z}}

\newcommand{\F}{{\mathbb F}}

\newcommand{\lb}{\left(}
\newcommand{\rb}{\right)}

\begin{document}
\title[Divisors of Fourier coefficients]{Divisors of Fourier coefficients of two newforms}

\author[Kumar, Kumari]{Arvind Kumar and Moni Kumari}
\address{Einstein Institute of Mathematics, the Hebrew University of Jerusalem, Edmund
	Safra Campus, Jerusalem 91904, Israel.}
\email{arvind.kumar@mail.huji.ac.il}
\address{Department of Mathematics, Bar-Ilan University, Ramat Gan 52900, Israel.}
\email{moni.kumari@biu.ac.il}

\begin{abstract}	
	For a pair of distinct non-CM newforms of weights at least 2,  having rational integral Fourier coefficients $a_{1}(n)$ and $a_{2}(n)$, under GRH, we obtain an estimate for the set of primes $p$ such that	
	$$
	\omega(a_1(p)-a_2(p)) \le  [ 7k+{1}/{2}+k^{1/5}  ],
	$$
	where $\omega(n)$ denotes the number of distinct prime divisors of an integer $n$ and $k$ is the maximum of their weights.
	As an application, under GRH, we show that  the number of  primes  giving congruences between two such newforms is  bounded by $[ 7k+{1}/{2}+k^{1/5} ]$. We also obtain a multiplicity one result for newforms via congruences.
	
\end{abstract}

\maketitle

\section{Introduction and statement of the results}	
For an elliptic curve $E/ \Q $ and a prime $p$ of good reduction, let $N_p(E):=p+1-a(p)$ be the number of points of the reduction of $E$ modulo $p$. Assume that $E$ is not $\Q$-isogenous to an elliptic curve with torsion. Then Koblitz's conjecture \cite{kob} says that the number of primes $p\le X$  for which  $N_p(E)$  is prime is asymptotically equal to $C_E \frac{X}{(\log X)^2}$, where $C_E$ is a positive constant depending on $E$. In particular, $N_p(E)$ 
is prime infinitely often when $p$ runs over the set of primes. This conjecture is still open but there are many results towards this in the literature (see \cite{swa,swb}). Indeed, Koblitz's conjecture can be seen as a variant of the twin prime conjecture (for more details, see \cite{kob}).

Motivated by Koblitz's conjecture, Kirti Joshi \cite{jos} has  studied an analogous question for the quantities $N_p(f):=p^k+1-a(p)$, where $a(p)$ is the (integer)  $p$th Fourier coefficient of a newform $f \in S_k(N)$, the space of cusp forms of weight $k$ and level $N$. As mentioned by Joshi, for the  Ramanujan Delta function $\Delta\in S_{12}(1)$ we have 
$\omega(N_p(\Delta)) \ge 3 
$ for any $p\ge 5$, where $\omega(n)$ is the number of distinct prime divisors of an integer $n$. This shows that, in general, the obvious variant of the Koblitz's conjecture is not true for modular forms of higher weights. In fact, he shows that there exist infinitely many cusp forms $f_{k_i}$ (not necessarily an eigenform) of increasing weight $k_i$ and of level $1$ 
such that  $\omega (N_p(f_{k_i})) \ge 2$ for all primes $p$.

If $f$ is a  non-CM newform of weight $k\ge 4$
then in the same paper Joshi gives an estimate for the primes $p$ for which  $N_p(f)$ is an almost prime, i.e., has few prime divisors.  More precisely, under GRH and Artin's holomorphy conjecture, he uses a suitably weighted sieve due to Richert  to prove that 
\begin{equation}\label{joshi_1}
	| \{p \le X: \omega (N_p(f)) \le  [ 5k+ 1+\sqrt{\log k}  ]  \}| \gg \frac{X}{(\log X)^2},
\end{equation}
where $[~\cdot~]$ is the greatest integer function. He also proves a similar result for the function $\Omega  (N_p(f))$, where $\Omega(n)$ counts the number of prime divisors of $n$ with multiplicity. 

One can realize $N_p(f)$ as the difference of $p$th Fourier coefficients of the normalised Eisenstein series $E_k$ of weight $k$ level 1 and the newform $f$. Motivated from this observation, the main aim of this article is to study a natural generalization of Joshi's work, namely, we consider
the difference of Fourier coefficients of two distinct newforms of arbitrary weights and study an analogous estimate like \eqref{joshi_1} which has many interesting applications.
We do this by  using product Galois representations, a refined version of Chebotarev’s theorem due to Serre and a suitably weighted sieve due to Richert. We make use of the ideas of \cite{jos, swa,swb}. Note that because of Deligne's estimate for $a(p)$, the quantity $N_p(f)\neq 0$ for any prime $p$ whereas the difference of $p$th Fourier coefficients of two newforms may be zero and hence we must have to remove those primes. More precisely, we prove the following.	
\begin{thm}\label{main}	
	Let $f_1 \in S_{k_1}(N_1)$ and $f_2 \in S_{k_2}(N_2)$ be  non-CM newforms with integer Fourier coefficients $a_1(n)$ and $a_2(n)$, respectively of weights at least 2. We also assume that $f_1$ and $f_2$ are not character twists of each other if $k_1=k_2$. Put $k=\max\{k_1,k_2\}.$ Then under GRH, we have
	\begin{equation}\label{main_estimate}
		|\{p \le X: a_1(p)\neq a_2(p) ~ and~ \omega (a_1(p)-a_2(p)) \le  [ 7k+{1}/{2}+ k^{1/5}  ]  \}|\gg \frac{X}{(\log X)^2}.
	\end{equation}
\end{thm}	
We remark that if $k \ge 6$, then one can replace the term $k^{1/5}$ with $\sqrt{\log k}$ on the left side of the estimate \eqref{main_estimate} and the same lower bound holds which gives a better estimate. This can be achieved by taking $\lambda=1/ \sqrt{\log k}$ instead of $\lambda=1/ k^{1/5}$ in the proof of \thmref{main} and then following the same arguments. Indeed, it is also clear  from the proof that this estimate can be improved further for newforms of higher weights.	

We now state a few applications of our main result. Unless stated otherwise, throughout the paper we shall work with forms $f_1$ and $f_2$ as in \thmref{main}. We also assume that a newform is always normalised so that its first Fourier coefficient is 1.
An immediate consequence of \thmref{main} is the following.
\begin{cor}\label{infinite_prime}
	Let $f_1$ and $f_2$ be newforms as in \thmref{main}. Then under GRH there exist infinitely many primes $p$ such that $a_1(p)\neq a_2(p)$ and
	$$
	\omega (a_1(p)-a_2(p)) \le  [ 7k+{1}/{2}+k^{1/5}  ] .
	$$
\end{cor}
We now recall a multiplicity one result which says that if $a_1(p)=a_2(p)$ for all but finitely many primes $p$, then $f_1=f_2$. Rajan \cite{raj} has extensively generalized this result by proving that if $a_1(p)=a_2(p)$ for a set of primes $p$ of positive upper density, then $f_1$ is a character twists of $f_2$.  This is known as a strong multiplicity one result. Recently, in \cite{mp}, a variant of this result for normalised Fourier coefficients has been obtained.  In this direction, we prove in \propref{same_coefficient} that, under GRH, if   
\begin{equation}\label{same_coefficients}
	|\{p\le X : a_1(p)=a_2(p)\}| \gg X^{13/14+\epsilon}
\end{equation}
for any $\epsilon>0$, then $f_1$ is a character twists of $f_2$.
As a consequence of \thmref{main}, we obtain the following interesting result that can be seen as a variant of a multiplicity one result in terms of congruences.
\begin{cor}\label{multiplicity_one}	
	Let $f_1$ and $f_2$ be non-CM normalised newforms of weight $k_1$ and $k_2$ with integer Fourier coefficients $a_1(n)$ and $a_2(n)$, respectively. Put $k=\max\{k_1,k_2\}$ and assume GRH. If there exist primes $\ell_1, \ell_2, \dots \ell_n$ such that $n >  [ 7k+{1}/{2}+k^{1/5} ] $ and for each $1\le i\le n$
	\begin{equation}\label{multi}
		a_1(p)\equiv a_2(p) \pmod {\ell_i},	
	\end{equation}	
	for all $p$ except for a set of primes of order $o\left(\frac{X}{(\log X)^2}\right)$, then $k_1=k_2$ and $f_1$ is a character twists of $f_2$.
\end{cor}
\begin{proof}
	On the contrary, assume that $f_1$ is not a character twists of $f_2$.
	For $1\le i\le n$, let $B_i(X)= \{ p\le X: a_1(p)\not\equiv a_2(p) \pmod {\ell_i}\}$. Put $B(X)=\cup_{i=0}^n B_i(X)$. Then for $p\notin B(X)$
	$$
	\ell_1\ell_2\dots \ell_n |(a_1(p)-a_2(p)) \implies \omega(a_1(p)-a_2(p))\ge n.
	$$
	In particular, 
	$$
	\{p\le X: a_1(p)\neq a_2(p) ~{\rm and}~\omega(a_1(p)-a_2(p))\le  [ 7k+{1}/{2}+k^{1/5}  ] \} \subset B(X).
	$$
	But from our assumptions in
	\eqref{multi} we have $B(X)=o\left(\frac{X}{(\log X)^2}\right)$
	and this contradicts \thmref{main}.
\end{proof}

We now mention the last application  of \thmref{main} which is related to the number of congruence primes of a newform. Recall that for a newform $f_1\in S_k(N_1)$ with integer Fourier coefficients $a_1(n)$, a positive integer $D$ is called a congruence divisor if there exists another newform $f_2\in S_k(N_2)$  with integer Fourier coefficients $a_2(n)$  which is not a character twists of $f_1$ such that $f_1$ and $f_2$ are congruent modulo $D$, i.e., $a_1(n)\equiv a_2(n)\pmod D$ for all $(n,N_1N_2)=1$. Indeed, this is equivalent to the condition that 
$a_1(p)\equiv a_2(p)\pmod D$
for all $(p,N_1N_2)=1$. Moreover, if $D$ is a prime it is called a congruence prime and we refer  \cite{gha} for a nice overview of the subject.    A congruence divisor of a newform is an important object to study as it is connected to many well-known problems. To name a few,  a bound of the largest congruence divisor is related to the ABC conjecture, and if $k=2$, then the congruence primes for $f_1$ are related to the prime divisors of the minimal degree of the modular parametrization to the elliptic curve attached to $f_1$ via the Eichler-Shimura mapping (see \cite{ram_murty}).  It would be also of great interest to bound the number of congruence primes of a newform (cf.  Remark on page no. 180 of \cite{ram_murty}).  
However, if we fix two newforms, then  the following result gives a bound on the number of congruence primes which  is immediate by \corref{infinite_prime}. 
\begin{cor}\label{congruence_divisor}
	Let $f_1$ and $f_2$ be newforms as in \thmref{main}. Suppose there exists a positive integer $D$ such that $f_1$ and $f_2$ are congruent modulo $D$. Then under GRH
	$$
	\omega(D) \le  [ 7k+{1}/{2}+k^{1/5}  ].
	$$
\end{cor}
Since each prime divisor of $D$ gives a congruence between $f_1$ and $f_2$, hence \corref{congruence_divisor} ensures that the number of primes giving congruences between two newforms is bounded uniformly only on their weights and not on the levels. This is the novelty of this result.

We now discuss some results about the function $\Omega (a_1(p)-a_2(p))$, where $p$ varies over the set of primes. Using a similar idea as the proof of \thmref{main}, we obtain the following.
\begin{thm}\label{Omega}
	Let $f_1$ and $f_2$ be as in \thmref{main}. Then under GRH, we have
	\begin{equation}\label{Omega_lower_bound}
		|\{p \le X: a_1(p)\neq a_2(p) ~ and~ \Omega (a_1(p)-a_2(p)) \le  [ 13k+{1}/{2}+\sqrt{\log k}   ]  \}|\gg \frac{X}{(\log X)^2}.
	\end{equation}
\end{thm}
It is clear that \thmref{Omega} also has applications of similar nature to that of \thmref{main} mentioned above and we would not repeat here.   
\begin{rmk}\label{upper_bound}
	It is possible to obtain an upper bound of the right order of magnitude for estimate \eqref{Omega_lower_bound}. In fact, we can do so by using Selberg’s sieve and the ideas used in the proof of \cite[Theorem 2.3.1]{jos}. More precisely, under GRH, one can obtain that if $f_1$ and $f_2$ are as in \thmref{main}, then 	
	\begin{equation}\label{Omega_upper_bound}
		|\{p \le X: a_1(p)\neq a_2(p) ~ and~ \Omega (a_1(p)-a_2(p)) \le  [(29k-{13})/{2}] \}| \ll \frac{X}{(\log X)^2}.	
	\end{equation}
\end{rmk}
From \eqref{Omega_upper_bound}, it follows that under GRH
$$
|\{p \le X:  a_1(p)-a_2(p) ~is ~prime \}| \ll \frac{X}{(\log X)^2}.
$$
In particular,  the natural density of the set $\{p :  a_1(p)-a_2(p) ~is ~prime \}$ is zero. It would be interesting to obtain a suitable lower bound of this set or at least to know whether there are infinitely many primes $p$ for which $a_1(p)-a_2(p)$ is a prime.

In fact, all the above results are valid even if we replace $a_1(p)-a_2(p)$ with $a_1(p)+a_2(p)$. Also,  similar results but with better bounds hold in \thmref{main} and \thmref{Omega} if we assume Artin’s holomorphicity conjecture in addition to GRH. 
It is also worth mentioning that the full strength of GRH is not essential to prove our theorems. Rather, a quasi-GRH, which assumes a zero-free region for the associated zeta functions in some fixed half-plane to the left of Re$(s) = 1$ is sufficient for our purpose.

\subsection*{Contents and structure of the paper} 
The theorem of Deligne connecting the theory of $\ell$-adic Galois representations to Fourier coefficients of newforms open the door for obtaining many new results regarding the arithmetical nature of these coefficients. This theorem plays an important role in the paper.
To prove our results, we first establish  \propref{asymptotic_pi*}
which gives an asymptotic formula for the number of primes $p$ up to $X$ for which $a_1(p)\neq a_2 (p)$ and  $a_1(p)-a_2 (p)$ is divisible by a fixed positive integer. We use  Galois representations attached to modular forms and
Chebotarev density theorem which is recalled in Section \ref{S_basic}. Proof of  \propref{asymptotic_pi*} requires computations of  the image of the product Galois representations attached to forms $f_1$ and $f_2$ and this is obtained in Section \ref{S_technical}. 
Finally, we apply a suitably weighted sieve due to Richert, recalled in Section \ref{S_sieving}, to prove our results. To establish the sieve conditions with the required uniformity of parameters \propref{asymptotic_pi*} plays a crucial role.
We use the ideas employed in \cite{swa, swb, jos} to prove our main results in Sections \ref{S_proof_main} and \ref{S_proof_Omega}.

\subsection*{Notation and conventions}
By GRH, we mean the Generalized 
Riemann Hypothesis; i.e., Riemann Hypothesis for all Artin $L$-functions. For 
any real number $X\ge 2$,  $\pi(X)$ denotes the number of primes less than or equal to $X$.
Along with the standard analytic notation $\ll, \gg, O, o, \sim$ (the  implied constants will often depend on the pair of forms under consideration), we use the letters
$p, \ell, q, \ell_1, \ell_2$ etc. to denote prime numbers throughout the paper. 

\section{Preliminaries}\label{S_basic}
In this section, we summarize some standard results without proofs which will be used throughout the paper. We closely follow \cite{gkk} for our exposition.
\subsection{Chebotarev density theorem}
We recall the Chebotarev density theorem which is one of the principal tools needed for proving
the main theorems of this paper.

Let $K$ be a finite Galois extension of $\Q$ with the Galois group $G$ and degree $n_K$. 
For an unramified prime $p$, we denote by ${\rm Frob}_p$, a Frobenius element of $K$ at $p$ in $G$.
For a subset $C$ of $G$, stable under conjugation, we define 
$$ 
\pi_C(X):=\{p\le X: p~ {\rm unramified ~in}~ K ~{\rm and ~Frob}_p\in C\}.
$$
The Chebotarev density theorem states that
$$
\pi_C(X)\sim \frac{|C|}{|G|}\pi(X).
$$
We will use the following conditional effective version of this theorem which was first obtained by Lagarias and Odlyzko \cite{lo} and was subsequently refined by Serre \cite{ser}. To state this, let $d_K$ be the absolute value of the discriminant of $K/\Q$ and  $\zeta_K(s)$ be the Dedekind zeta function associated with $K$. 
\begin{proposition}\label{effective_cdt}
	Suppose $\zeta_K(s)$ 	satisfies the Riemann hypothesis. Then
	\begin{equation*}
		\pi_C(X)= \frac{|C|}{|G|}\pi(X)+O\bigg(\frac{|C|}{|G|}X^{{1}/{2}}(\log d_K+n_K\log X)\bigg).
	\end{equation*}
\end{proposition}
By assuming, in addition to GRH, Artin’s holomorphy conjecture, one can improve the
error term in the above asymptotic formula for 	$\pi_C(X)$.

\subsection{mod-$h$ Galois representations} 
Let $G_{\Q}= {\rm Gal}(\bar\Q/\Q)$ be the absolute Galois group of an algebraic closure $\bar\Q$ of $\Q$. Let $k\ge 2, N\ge 1$ and $ \ell $ be a prime. Suppose $f\in S_k(N)$  is a  newform
with integer Fourier coefficients $a (n)$. Work of Shimura, Deligne and Serre gives the existence of a 
two-dimensional  continuous, odd and irreducible Galois representation
\begin{equation*}
	\rho_{f, \ell }:G_{\Q}\rightarrow {\rm GL}_2(\Z_\ell )
\end{equation*}
which is unramified at $p\nmid N\ell$.  If ${\rm{Frob}}_p$ denotes a Frobenius element corresponding to such a prime, then the representation $\rho_{f, \ell }$ has the property that 
$$ {\rm{tr}}(\rho_ {f,\ell}({\rm{Frob}}_p))=a (p), ~~~~ {\rm {det}}(\rho_ {f,\ell}({\rm{Frob}}_p))=p^{k-1}.$$
By reduction and semi-simplification, we obtain a mod-$\ell$ Galois representation, namely 
$$ {\overline{\rho}}_ {f,\ell}:G_{\Q}\rightarrow {\rm GL}_2(\F_\ell), $$
where $\F_{\ell}:=\Z/\ell\Z$.

Let $h=\prod_{j=1}^{t}\ell_j^{n_j}$ be a positive integer.
Using the $\ell_j$-adic representations attached to $f$, 
we consider an $h$-adic representation
\begin{equation*}
	{\rho}_{f,h}:  G_\Q \rightarrow {\rm GL}_2\big(\prod_{1\le j\le t}\Z_{\ell_j}\big).
\end{equation*}
For each $1\le j\le t$, we have the natural projection $\Z_{\ell_j}\twoheadrightarrow {\Z/{\ell_j^{n_j}}\Z}$, 
and hence 
we obtain a mod-$h$ Galois representation given by 

$$ {\overline{\rho}}_{f,h} :G_\Q \rightarrow
{\rm GL}_2(\prod_{1\le j\le t}{{\Z/{\ell_j^{n_j}}\Z}}) \xrightarrow{\cong} {\rm GL}_2(\Z/h\Z).$$
Furthermore, if $p\nmid Nh$ is a prime, then $\bar{\rho}_{f,h}$ 
is unramified at $p$ and
$$ {\rm tr}\left(\bar{\rho}_{f,h}\lb{\rm Frob}_p\rb\right)=a (p),~~~~ 
{\rm det}\left(\bar{\rho}_{f,h}\lb{\rm Frob}_p\rb\right)= p^{k-1}. $$

Let $f_1 \in S_{k_1}(N_1)$ and $f_2 \in S_{k_2}(N_2)$ be   newforms having integer Fourier coefficients $a_1(n)$ and $a_2(n)$, respectively.
Then one can consider the product 
representation $\bar{\rho}_h$ of $\bar{\rho}_{f_1,h}$ and $\bar{\rho}_{f_2,h}$, defined by
\begin{align*}
	\bar{\rho}_h: G_\Q & \rightarrow {\rm GL}_2(\Z/h\Z)\times {\rm GL}_2(\Z/h\Z),\\
	\sigma & \mapsto (\bar{\rho}_{f_1,h}(\sigma), \bar{\rho}_{f_2,h}(\sigma)) \notag.
\end{align*}
Let $\mathscr{A}_h$ denote the image of $G_\Q$ under $\bar{\rho}_h$.
By the fundamental theorem of Galois theory,  the fixed field of ${\rm ker}(\bar\rho_h)$, say $L_h$, is a finite 
Galois extension of $\Q$ and
\begin{equation}\label{image}
	{\rm Gal}(L_h/ \Q) \cong \mathscr{A}_h.
\end{equation}
Let $\mathscr{C}_h$ be the subset of $\mathscr{A}_h$ defined by
\begin{equation*}\label{C_h}
	\mathscr{C}_h=\{(A,B)\in \mathscr{A}_h: {\rm tr}(A)={\rm tr}(B)\}.
\end{equation*}
We now define the following function on the set of positive integers which will play an important role throughout the paper. For an integer $h >1$, define
\begin{equation}\label{delta_definition}
	\delta(h):=\frac{|\mathscr{C}_h|}{|\mathscr{A}_h|}
\end{equation}
and $\delta(1):=1$. 
Since the trace of the 
image of complex conjugation is always zero, 
$\mathscr{C}_h\neq \phi$, and hence $\delta(h)>0$ for every integer $h$. 

\section{Technical results}\label{S_technical}
Let   $f_1$ and $f_2$ be newforms as before. The main aim of this section is to obtain an asymptotic size of  $\delta(\ell^n)$ for $n=1,2$ and this require the computation of the cardinalities of  $\mathscr{A}_{\ell^n}$ and $\mathscr{C}_{\ell^n}$.
Building on the work of Ribet \cite{rib} and Momose \cite{mom},  Loeffler \cite[Theorem 3.2.2]{loe} has determined the image $\mathscr{A}_{\ell^n}$ of the product Galois representations. It follows that there exists a 
positive constant $M(f_1, f_2)$ such that for all primes $\ell \ge M(f_1, f_2)$ and $n\ge 1$
\begin{align}\label{loeffler}
	\mathscr{A}_{\ell^n}\!=\!\{(A, B)\!\in \!{\rm GL}_2(\Z/\ell^n\Z)\!\times\! {\rm GL}_2(\Z/\ell^n\Z)\!: 
	{\rm det}(A)\!=v^{k_1 -1}\!, {\rm det} (B)\!= v^{k_2 -1}\!, v \!\in\! {(\Z/\ell^n\Z)^{\times}} \}.
\end{align}
In other words, the mod-$\ell^n$ representations of two newforms (that are not character twists of each other)
are as independent as possible. In the rest of the paper, we denote the constant $M(f_1, f_2)$ by $M$ and without loss of generality, we assume that $M\ge 3$. 
Clearly, for $\ell \ge M$
\begin{align}\label{C_h_definition}
	\mathscr{C}_{\ell^n} = &\{(A, B) \in\mathscr{A}_{\ell^n}: 
	{\rm tr}(A)={\rm tr}(B) \}.
\end{align}

\subsection{Combinatorial lemmas} 
Here we obtain results about cardinalities of  $\mathscr{A}_{\ell^n}$ and $\mathscr{C}_{\ell^n}$. 
We first assume that
$$
\lambda_n=gcd~(\ell^{n}-\ell^{n-1},k_1-1,k_2-1)
$$
and
\begin{equation}\label{Lambda}
	\Lambda_n=\{ (v^{k_1-1},v^{k_2-1}): v\in \mathbb (\Z/{\ell^n}\Z)^{\times}  \}.
\end{equation}
We now consider the group homomorphism 
\begin{equation*}
	\phi:  \mathbb (\Z/{\ell^n}\Z)^{\times}\rightarrow \Lambda_n
	{\rm ~defined ~by} ~\phi(v)=  (v^{k_1-1},v^{k_2-1}).
\end{equation*}
Since $\phi$	is surjective and its kernel $\{v \in \mathbb (\Z/{\ell^n}\Z)^{\times}: v^{\lambda_n}=1\}$ is a cyclic subgroup of $\mathbb (\Z/{\ell^n}\Z)^\times$ of order $\lambda_n$, we obtain
\begin{equation}\label{Lambda_cardinality}
	|\Lambda_n|=\frac{| (\Z/{\ell^n}\Z)^{\times} |}{\lambda_n}
	=\frac{\ell^n-\ell^{n-1}}{\lambda_n}.
\end{equation}

We first recall the following result proved in \cite[Lemma 3.3]{gkk}. 
\begin{lemma}\label{A_l_computation}
	For  any prime $\ell \ge M$ 
	$$ | \mathscr{A}_{\ell} | =\frac{1}{\lambda_1}(\ell-1)^3(\ell^2+\ell)^2.$$
\end{lemma}
Using \lemref{A_l_computation}, we now compute $| \mathscr{A}_{\ell^n} | $ for any $n\ge 1$.
\begin{lemma}\label{A_ln_computation}
	For  any prime $\ell \ge M$ and integer $n \ge 1$
	$$ | \mathscr{A}_{\ell^n} | =\frac{1}{\lambda_n}\ell^{7(n-1)}(\ell-1)^3(\ell^2+\ell)^2.$$
\end{lemma}
\begin{proof}
	Let $\psi:  \mathscr{A}_{\ell^n} \rightarrow  \mathscr{A}_{\ell}$ be the natural reduction map. Since it is a surjective group homomorphism, we have
	$$
	|\mathscr{A}_{\ell^n} |= |{\rm ker}(\psi)|~ |\mathscr{A}_{\ell} |.
	$$
	Therefore, in view of \lemref{A_l_computation}, to evaluate  $|\mathscr{A}_{\ell^n}|$ it is sufficient to compute $|{\rm ker}(\psi)|$. For that we first record the following result that can be proved easily:  for a given $d \in \mathbb (\Z/{\ell^n}\Z)^{\times}$  with $d\equiv 1 \pmod \ell$
	\begin{equation}\label{ca}
		|\{  \gamma\in {\rm GL}_2(\Z/\ell^n\Z): {\rm det}(\gamma)= d, \gamma\equiv {\rm Id}\pmod \ell 
		\}| =\ell^{3(n-1)},
	\end{equation}
	where Id is the identity element in ${\rm GL}_2(\F_\ell)$.
	Now, we note that 
	\begin{align*}
		{\rm ker}(\psi)= \{(A, B)\in \mathscr{A}_{\ell^n}:   (A,B)\equiv ({\rm Id, Id})\pmod \ell \}|,
	\end{align*}
	therefore from \eqref{loeffler}
	\begin{align*}
		|{\rm ker}(\psi)|&= \sum_{(d_1,d_2)\in \Lambda_n} \sum_{\substack{A\in {\rm GL}_2(\Z/\ell^n\Z)\\ {\rm {det}}(A)=d_1, A\equiv {\rm Id}\pmod \ell }}1 \sum_{\substack{B\in {\rm GL}_2(\Z/\ell^n\Z)\\ {\rm {det}}(B)=d_2, B\equiv{\rm Id} \pmod \ell }}1.
	\end{align*}
	In the above, congruence conditions on $A$ and $B$  compel that $d_1\equiv d_2\equiv 1\pmod \ell$ and hence using \eqref{ca} gives
	$$
	|{\rm ker}(\psi)|= \ell^{6(n-1)}\sum_{\substack{(d_1,d_2)\in \Lambda_n\\
			d_1\equiv d_2\equiv 1 \pmod \ell}}1.
	$$
	Since the sum appearing on the right side of the above is the cardinality of the kernel of  the natural (surjective) reduction map $ \Lambda_n\rightarrow \Lambda_1$, therefore
	$$
	|{\rm ker}(\psi)|= \frac{|\Lambda_n|}{|\Lambda_1|} \ell^{6(n-1)}.
	$$
	Now using   \eqref{Lambda_cardinality} in the above yields the desired result.
\end{proof}

Our next aim is to compute the cardinalities of $\mathscr{C}_{\ell}$ and $\mathscr{C}_{\ell^2}$. Though an explicit computation  is possible, we only obtain asymptotic formulas here and that is enough for our purpose.
To simplify our notation, we  denote the set of quadratic  and non-quadratic residue elements in $(\Z/{\ell^n}\Z)^{\times}$ by $Q_n$ and $Q_n^c$, respectively. 

\begin{lemma}\label{C_l_computation_general}
	For any prime $\ell \ge M$,
	$$ | \mathscr{C}_{\ell} | =\frac{\ell^6}{\lambda_1}+O(\ell^5).$$
\end{lemma}

\begin{proof}
	
	From the definition of  $\mathscr{C}_{\ell}$
	\begin{align*}
		|\mathscr{C}_{\ell}|
		&=  \sum_{(d_1,d_2)\in \Lambda_1} |\{(A,B)\in {\rm GL}_2({\F_\ell})\times {\rm GL}_2({\F_\ell}): {\rm {det}}(A)=d_1, {\rm {det}}(B)=d_2, {\rm tr}(A)={\rm tr}(B)\}|\notag\\
		&=\sum_{t \in \mathbb F_{\ell}}  \sum_{(d_1,d_2)\in  \Lambda_1}   \sum_{\substack{A\in {\rm GL}_2(\F_\ell)\\ {\rm {det}}(A)=d_1, {\rm tr}(A)=t}}1 \sum_{\substack{B\in {\rm GL}_2({\F_\ell})\\ {\rm {det}}(B)=d_2, {\rm tr}(B)=t}}1.
	\end{align*}	
	Split the sum over $\Lambda_1$ 
	into three parts, namely  
	\begin{align}\label{three}
		|\mathscr{C}_{\ell}|
		&= \sum_{t \in \mathbb F_{\ell}}  \bigg[ \sum_{\substack{(d_1,d_2)\in  \Lambda_1\\ t^2-4d_1\in Q_1}}+ 
		\sum_{\substack{(d_1,d_2)\in  \Lambda_1\\ t^2=4d_1}}+ \sum_{\substack{(d_1,d_2)\in  \Lambda_1\\ t^2-4d_1\in Q_1^c}} \bigg] \sum_{\substack{A\in {\rm GL}_2(\mathbb F_\ell)\\ {\rm {det}}(A)=d_1, {\rm tr}(A)=t}}1 \sum_{\substack{B\in {\rm GL}_2(\mathbb F_\ell)\\ {\rm {det}}(B)=d_2, {\rm tr}(B)=t}}1
	\end{align}
	and we denote the corresponding sums by $S_1$, $S_2$ and $S_3$, respectively. Thus
	\begin{equation*}
		S_1=\sum_{t \in \mathbb F_{\ell}}  \sum_{\substack{(d_1,d_2)\in  \Lambda_1\\ t^2-4d_1\in Q_1}}  \sum_{\substack{A\in {\rm GL}_2(\mathbb F_\ell)\\ {\rm {det}}(A)=d_1, {\rm tr}(A)=t}}1 \sum_{\substack{B\in {\rm GL}_2(\mathbb F_\ell)\\ {\rm {det}}(B)=d_2, {\rm tr}(B)=t}}1.
	\end{equation*}	
	To proceed further, note that for 
	given $d \in\mathbb F_\ell^{\times}$ and $t \in \mathbb F_{\ell}$ one can obtain the following result by employing the elementary counting arguments.
	\begin{align}\label{id11}
		&|\{\gamma \in {\rm GL}_2(\mathbb F_\ell): ~{\rm {det}}(\gamma)=d, {\rm tr}(\gamma)=t \}| =
		\begin{cases}
			\ell^2+\ell, & t^2-4d\in Q_1,\\
			\ell^2, & t^2=4d,\\
			\ell^2-\ell, &  t^2-4d\in Q_1^c. 
		\end{cases}
	\end{align}
	Using  \eqref{id11} gives
	\begin{align*}
		S_1&=(\ell^2+\ell) \sum_{t \in \mathbb F_{\ell}}  \sum_{\substack{(d_1,d_2)\in  \Lambda_1\\ t^2-4d_1\in Q_1}} 
		\sum_{\substack{B\in {\rm GL}_2(\mathbb F_\ell)\\ {\rm {det}}(B)=d_2, {\rm tr}(B)=t}}1\\	
		&= (\ell^2+\ell) \sum_{t \in \mathbb F_{\ell}}  \bigg[ \sum_{\substack{(d_1,d_2)\in  \Lambda_1\\ t^2-4d_1\in Q_1\\  t^2-4d_2\in Q_1}} +  \sum_{\substack{(d_1,d_2)\in  \Lambda_1\\ t^2-4d_1\in Q_1\\  t^2=4d_2}} + \sum_{\substack{(d_1,d_2)\in  \Lambda_1\\ t^2-4d_1\in Q_1\\  t^2-4d_2\in Q_1^c}} 
		\bigg] \sum_{\substack{B\in {\rm GL}_2(\mathbb F_\ell)\\ {\rm {det}}(B)=t_2, {\rm tr}(B)=t}}1.
	\end{align*}
	Again using  \eqref{id11}
	\begin{align*}
		S_1
		&= (\ell^2+\ell) \sum_{t \in \mathbb F_{\ell}}  \bigg[(\ell^2+\ell) \sum_{\substack{(d_1,d_2)\in  \Lambda_1\\ t^2-4d_1\in Q_1\\  t^2-4d_2\in Q_1}} 1+ \ell^2 \sum_{\substack{(d_1,d_2)\in  \Lambda_1\\ t^2-4d_1\in Q_1\\  t^2=4d_2}} 1+ (\ell^2-\ell) \sum_{\substack{(d_1,d_2)\in  \Lambda_1\\ t^2-4d_1\in Q_1\\  t^2-4d_2\in Q_1^c}} 1
		\bigg].
	\end{align*}
	Collecting the terms containing  $\ell^4$ gives
	\begin{align*}
		S_1&= \ell^4 \sum_{t \in \mathbb F_{\ell}}  \sum_{\substack{(d_1,d_2)\in  \Lambda_1\\ t^2-4d_1\in Q_1}}  1+ O(\ell^5).
	\end{align*}
	Similarly, we have  
	$$
	S_2= \ell^4 \sum_{t \in \mathbb F_{\ell}}  \sum_{\substack{(d_1,d_2)\in  \Lambda_1\\ t^2=4d_1}}   1
	+ O(\ell^5)
	$$
	$$
	S_3= \ell^4 \sum_{t \in \mathbb F_{\ell}}  \sum_{\substack{(d_1,d_2)\in  \Lambda_1\\ t^2-4d_1\in Q_1^c}}1
	+ O(\ell^5).
	$$
	Combining all together, we have, from \eqref{three}
	\begin{align*}
		|\mathscr{C}_{\ell}|
		&=\ell^4 \sum_{t \in \mathbb F_{\ell}}  \sum_{(d_1,d_2)\in  \Lambda_1} 1 +O(\ell^5)
	\end{align*}
	and now using \eqref{Lambda_cardinality} completes the proof.
\end{proof}
To compute $|\mathscr{C}_{\ell^2}|$, we first prove the  following result which is a generalisation of \eqref{id11} for the ring $\Z/\ell^2 \Z$.
\begin{lemma}\label{l2_computation}
	For any $d\in  (\Z/\ell^2\Z)^\times$ and $t \in  \Z/\ell^2\Z$, we have
	\begin{align}\label{id12}
		&|\{\gamma \in {\rm GL}_2(\Z/\ell^2\Z): ~{\rm {det}}(\gamma )=d, {\rm tr}(\gamma )=t \}| =
		\begin{cases}
			\ell^4+\ell^3-\ell^2, & t^2-4d=0,\\
			\ell^4-\ell^2, & 0\neq t^2-4d\equiv 0 \pmod \ell,\\
			\ell^4+\ell^3, & t^2-4d\in Q_2,\\
			\ell^4-\ell^3, &  t^2-4d\in Q_2^c. 
		\end{cases}
	\end{align}
\end{lemma}
\begin{proof} 
	It is clear that 
	$$|\{\gamma \in {\rm GL}_2(\Z/\ell^2\Z): ~{\rm {det}}(\gamma )=d, {\rm tr}(\gamma )=t \}| = |\mathscr N|,$$
	where $\mathscr N:=  \{(a,b,c)\in (\Z/\ell^2\Z)^3: a^2-at+bc=-d\}$.
	To compute $|\mathscr N|$ we divide the set $\mathscr N$ into three disjoint subsets $\mathscr N_1$, $\mathscr N_2$ and $\mathscr N_3$ based on the following three cases, respectively. Hence 
	\begin{equation}\label{idN}
		|\mathscr N| = |\mathscr N_1| + |\mathscr N_2|+ |\mathscr N_3|.
	\end{equation}
	{\bf Case (i): $a=0$.} Then the condition $bc=-d$  forces that  $b$ and $c$ both have to be units and for any $b$ there exists a unique $c$. Hence
	$$
	|\mathscr N_1|=\ell^2-\ell.
	$$
	{\bf Case (ii): $a\neq 0$ and $bc=0$.}  The latter condition implies that either $b$ or $c$ is 0, or both are (non-zero) zero-divisors of  $\Z/\ell^2\Z$.  The total number of such pairs is $2\ell^2-1+(\ell-1)^2=3\ell^2-2\ell$. Therefore,
	\begin{equation}\label{N2}
		|\mathscr N_2|= |\{a \in \Z/\ell^2\Z: a^2-at+d= 0\}| \times (3\ell^2-2\ell).
	\end{equation}
	We now  claim that
	\begin{equation}\label{solution}
		|\{a\in \Z/\ell^2\Z: a^2-at+d=0\} |=
		\begin{cases}
			\ell, & t^2-4d=0, \\
			0, &  0\neq  t^2-4d \equiv 0\pmod \ell,\\
			2, &  t^2-4d\in Q_2, \\
			0, & t^2-4d\in Q_2^c. \\
		\end{cases}
	\end{equation}
	To prove this, we see that if $t^2-4d=0$, then any $a\equiv \frac{t}{2} \pmod \ell$ is a solution of $a^2-at+d=0$ and there are $\ell$ such choices for $a$. Next, assume that $0\neq t^2-4d\equiv 0\pmod \ell$. If  $a^2-at+d=0$ has solutions, say $x$ and $y$, then $(x-y)^2=(x+y)^2-4xy =t^2-4d \equiv 0\pmod \ell$. Therefore, 
	$x-y \equiv 0 \pmod \ell \implies t^2-4d=(x-y)^2=0$ which is a contradiction. 
	The last two cases are clear. \\	
	Thus using \eqref{solution} in \eqref{N2} gives the cardinality of $\mathscr N_2$.\\
	{\bf Case (iii): $a\neq 0$ and $bc\neq0$.} In this case, $bc$ can be either a (non-zero) zero divisor or a unit. Clearly, the number of choices for $b$ and $c$ such that $bc$ is a given non-zero zero divisor is $2\ell(\ell-1)$ and for a given unit the number of such choices is $\ell^2-\ell$. Therefore, we have
	\begin{align}\label{case3}
		|\mathscr N_3|=& |\{a\in \Z/\ell^2\Z: 0\neq a^2-at+d\equiv 0\pmod \ell\}| \times 2\ell(\ell-1)\notag\\
		& +|\{a\in \Z/\ell^2\Z: a^2-at+d\in (\Z/\ell^2\Z)^\times\}|\times  (\ell^2-\ell).
	\end{align}
	If $a^2-at+d=m\ell$ for some $m\in \F_{\ell}^\times$, then from  \eqref{solution} 
	\begin{equation*}
		|\{a\in \Z/\ell^2\Z: a^2-at+d=m\ell \}|=
		\begin{cases}
			\ell, & t^2-4(d-m\ell)=0, \\
			0, &  0\neq  t^2-4(d-m\ell) \equiv 0 \pmod \ell,\\ 
			2, &  t^2-4(d-m\ell)\in Q_2\iff   t^2-4d\in Q_2, \\
			0, & t^2-4(d-m\ell)\in Q_2^c \iff   t^2-4d\in Q_2^c. \\
		\end{cases}
	\end{equation*}
	Note that there exists a unique  $m\in \F_{\ell}^\times$ such that   $t^2-4(d-m\ell)=0$ and in that case $0\neq t^2-4d\equiv 0\pmod \ell$. Therefore 
	\begin{equation}\label{solution1}
		|\{a\in \Z/\ell^2\Z: 0\neq a^2-at+d\equiv 0 \pmod \ell \}|=
		\begin{cases}
			0, & t^2-4d=0, \\
			\ell, &  0\neq  t^2-4d \equiv 0 \pmod \ell,\\
			2(\ell-1), &  t^2-4d\in Q_2, \\
			0, & t^2-4d\in Q_2^c. \\
		\end{cases}
	\end{equation}
	As we have $\ell^2-1$ choices of $a$ in this case, \eqref{solution} and \eqref{solution1} immediately gives 
	\begin{equation}\label{solution2}
		|\{a\in \Z/\ell^2\Z: a^2-at+d\in (\Z/\ell^2\Z)^\times \}|=
		\begin{cases}
			\ell^2-\ell-1, & t^2-4d=0, \\
			\ell^2-\ell-1, &  0\neq  t^2-4d \equiv 0 \pmod \ell,\\
			\ell^2-2\ell-1, &  t^2-4d\in Q_2, \\
			\ell^2-1, & t^2-4d\in Q_2^c. \\
		\end{cases}
	\end{equation}
	Substituting \eqref{solution1} and \eqref{solution2} in \eqref{case3} and then combining all the above three cases in \eqref{idN} gives the desired result.	 
\end{proof}

Using \lemref{l2_computation} and following a similar argument as in the proof of \lemref{C_l_computation_general}, we obtain:
\begin{lemma}\label{C_l2}
	For any prime $\ell \ge M$,
	$$ | \mathscr{C}_{\ell^2} | =\frac{\ell^{12}}{\lambda_2}+O(\ell^{11}).$$
\end{lemma} 
Let $h=\ell_1^{n_1}\ell_2^{n_2}\dots \ell_t^{n_t}$. 
Since the fixed field of ${\rm ker}(\bar\rho_h)$ is contained in the compositum of fixed fields of ${\rm ker}(\bar\rho_{\ell_i^{n_i}})$, from \eqref{image}
$$
|\mathscr{A}_h|\le   |\mathscr{A}_{\ell_1^{n_1}}|~ |\mathscr{A}_{\ell_2^{n_2}}|\dots  |\mathscr{A}_{\ell_t^{n_t}}|
\hspace{20pt} {\rm and}\hspace{20pt}
|\mathscr{C}_h|\le  |\mathscr{C}_{\ell_1^{n_1}}|~ |\mathscr{C}_{\ell_2^{n_2}}|\dots  |\mathscr{C}_{\ell_t^{n_t}}|.
$$
Since for any prime $\ell$ and integer $n\ge 1$, $\mathscr{A}_{\ell^n}$ is contained in the set
\begin{align*}
	\{(A, B)\in {\rm GL}_2(\Z/\ell^n\Z)\times {\rm GL}_2(\Z/\ell^n\Z): 
	{\rm det}(A)=v^{k_1 -1}, {\rm det} (B)= v^{k_2 -1}, v \in {(\Z/\ell^n\Z)^{\times}} \},
\end{align*}
hence a simple counting argument gives
\begin{equation*}\label{aln}
	|\mathscr{A}_{\ell^n}|\ll \ell^{7n} \hspace{20pt} {\rm and }\hspace{20pt}
	|\mathscr{C}_{\ell^n}|\ll \ell^{6n}.
\end{equation*}
Therefore now it is clear that for any integer $h\ge 1$
\begin{equation}\label{upper_bound_C_h}
	|\mathscr{A}_h|\ll h^7
	\hspace{20pt} {\rm and }\hspace{20pt}
	|\mathscr{C}_h|\ll {h^6}.
\end{equation}


\subsection{Asymptotic size of  $\delta(\ell)$}
Recall that for any positive integer $h> 1$, 
$$\delta(h)=\frac{|\mathscr{C}_h|}{|\mathscr{A}_h|}.
$$
An immediate consequence of the results in the previous section, we have  the following.
\begin{proposition}\label{asymptotic_delta}
	As
	$\ell$ varies over
	primes then for  $n =1,2$
	\begin{equation*}
		\delta(\ell^n)\sim \frac{1}{\ell^n} {~  as ~  } \ell \rightarrow \infty.
	\end{equation*}
\end{proposition}
Using the explicit description of $\mathscr{A}_\ell$ and $\mathscr{C}_\ell$ given in  \eqref{loeffler} and \eqref{C_h_definition}, one can easily prove that 
if $f_1$ and $f_2$ are newforms as before then the following holds.
\begin{proposition}\label{delta_multiplicative_large_prime}
	For  primes $\ell_1, \ell_2>M$ with $\ell_1 \neq \ell_2$, we have
	$$
	\delta(\ell_1 \ell_2)=\delta(\ell_1) \delta(\ell_2).
	$$
\end{proposition}

\section{Analytic results on primes}\label{S_sum}
Recall that  $f_1$ and $f_2$  are non-CM newforms with integer Fourier coefficients which are not character twists of each other. For a positive integer $h\ge 1$
and a real number $X\ge 2$, consider the function
\begin{equation}\label{pi_definition}
	\pi_{f_1,f_2}(X,h):=\sum_{\substack{p\le X, (p, h N)=1 \\ h|(a_1 (p)-a_2(p))}}1.
\end{equation}

The representation $\bar\rho_h$, defined in Section \ref{S_basic}, is unramified outside $hN$. Also, it is ramified at all the primes $\ell |h$ because its determinant constituents a non-trivial power of the mod $\ell$ cyclotomic character which is ramified at $\ell$. However, there may exist some primes dividing $N$ at which $\bar\rho_h$ is unramified. 
It follows that a prime $p$ is 
unramified in $L_h$ only if either $(p,hN)=1$ or $p|N$. Since the image of Frobenius elements under $\bar\rho_h$ generate $\mathscr{A}_h$, we can write
\begin{align*}
	\pi_{f_1,f_2}(X,h)
	&= |\{p\le X: p {\rm ~unramified~in~} L_h, \bar{\rho}_h\left({\rm Frob}_p \right) \in \mathscr{C}_h \}|+O(1),
\end{align*}
where the error term is due to the possible primes divisors of $N$ which are unramified in $L_h$. Now applying the Chebotarev density theorem (see \propref{effective_cdt}) for the field $L_h$, the group $\mathscr{A}_h$ 
and the set $\mathscr{C}_h$ which is stable under conjugation, we obtain the following. 
\begin{proposition}\label{asymptotic_pi}
	Let $f_1\in S_{k_1}(N_1)$ and $f_2\in S_{k_2}(N_2)$ be non-CM newforms with rational integral
	coefficients $a_1(n)$ and $a_2(n)$, respectively. Assume that $f_1$ and $f_2$ are not character twists of each other. Let $N=lcm(N_1,N_2)$ and  $h\ge 1$ be an integer. 
	If GRH is assumed, then for any positive integer $h$
	\begin{equation}\label{evl}
		\pi_{f_1,f_2}(X,h)=\delta(h){\pi(X)}+O\left(h^6X^{{1}/{2}} \log(h NX)\right).
	\end{equation}
\end{proposition}
Here we need to use \eqref{upper_bound_C_h} and the following variation of a result of Hensel (see \cite[Proposition 5, p. 129]{ser}).
\begin{equation}
	\log d_{L_h}\le \mathscr{A}_h \log(hN \mathscr{A}_h).
\end{equation} 

For our purpose, we now use \propref{asymptotic_pi} to obtain the following result giving an upper bound for the set of primes $p$ with  $a_1(p)=a_2(p)$. This may be also of independent interest.
\begin{proposition}\label{same_coefficient}
	Let $f_1$ and $f_2$ be newforms as before. Then under GRH
	$$
	|\{p\le X : a_1(p)=a_2(p)\} |=O(X^{13/14}).
	$$
\end{proposition}
\begin{proof}
	Clearly, for any prime  $\ell$
	$$
	|\{p\le X: a_1(p)=a_2(p)\}| \le \pi_{f_1,f_2}(X,\ell) +O(1).
	$$
	Hence using \propref{asymptotic_pi}, for a large prime $\ell$ 
	$$
	|\{p\le X: a_1(p)=a_2(p)\}| =O\left( \frac{{\pi(X)}}{\ell} \right) +O\left(\ell^6X^{{1}/{2}} \log(\ell NX)\right).
	$$
	Now by  Bertrand's postulate, we chose a prime $\ell$ between $\frac{X^{1/14}}{\log X}$ and $2 \frac{X^{1/14}}{\log X}$ and this proves the result.
\end{proof}
We remark that for newforms of weight 2 and by making use of
various abelian extensions,  in \cite[Theorem 10]{mmp}, a better estimate in \propref{same_coefficient} is obtained.

We now define
\begin{equation}\label{pi*_definition}
	\pi_{f_1,f_2}^*(X,h)=\sum_{\substack{p\le X \\ h|(a_1 (p)-a_2(p)) \\ a_1(p)\neq a_2(p)} }1.
\end{equation}
Using \propref{asymptotic_pi} and \propref{same_coefficient} we deduce the following.
\begin{proposition}\label{asymptotic_pi*}
	Let $f_1$ and $f_2$ be newforms as in \propref{asymptotic_pi}. If GRH is assumed, then for any positive integer $h$ 
	\begin{equation}\label{me}
		\pi_{f_1,f_2}^*(X,h)=\delta(h)\pi(X)+O\left(h^6X^{{1}/{2}}\log (hNX)\right)+O(X^{13/14}).
	\end{equation}
\end{proposition}

\begin{remark}\label{sum_Fourier_coefficients}
	Indeed, the estimates given in  \propref{asymptotic_pi} and \propref{asymptotic_pi*}  are also valid for the set of primes $p\le X$ with $h|(a_1 (p)+a_2(p))$ (with an extra condition $a_1 (p)+a_2(p)\neq 0$ for the latter one). This can be achieved by considering the set $\mathscr{C}_h'=\{(A,B)\in \mathscr{A}_h: {\rm tr}(A)=-{\rm tr}(B)\} $ instead of   $\mathscr{C}_h$ in Section \ref{S_technical} and following the same arguments.
\end{remark}

\begin{remark}\label{artin}
	In the above propositions, if one assumes Artin's holomorphy conjecture in addition to GRH, then an improved error term can be obtained. 
	More precisely, in \propref{asymptotic_pi} and \propref{asymptotic_pi*}, we have  $O\left(h^3 X^{{1}/{2}}\log (hNX)\right)$ instead of $O\left(h^6X^{{1}/{2}}\log (hNX)\right)$ which gives the following estimate for \propref{same_coefficient}.
	\begin{equation}
		|\{p\le X : a_1(p)=a_2(p)\}| =O(X^{7/8}).
	\end{equation}
\end{remark}

\section{Sieving tool: Richert's weighted one-dimensional sieve form}\label{S_sieving}
We will prove \thmref{main} by using a suitably weighted sieve due to Richert. The sieve problem we encounter here is a one-dimensional sieve problem in the parlance of “sieve methods”. We will use notation and  conventions from \cite{HR}.

Let $\mathcal{A}$ be a finite set of integers not necessarily positive or distinct. Let $\mathcal{P}$
be an infinite set of prime numbers. For each prime $\ell \in \mathcal{P}$, 
let $\mathcal{A}_{\ell} :=\{a\in \mathcal{A}: a \equiv 0 \pmod \ell \}.$ We write
\begin{equation}\label{approximation}
	| \mathcal{A}|=X+r_1 \hspace{15pt} {\rm and} \hspace{15pt} | \mathcal{A}_{\ell}| =\delta{(\ell})X+r_{\ell},
\end{equation}
where $X$ (resp. $\delta({\ell})X$) and $r_1$ (resp. $r_{\ell})$ are a close  approximation and the remainder to $\mathcal{A}$ (resp. $\mathcal{A}_{\ell}$), respectively.
For a square free integer $d$ composed of primes of $\mathcal{P}$, let
$$\mathcal{A}_{d} =\{a\in \mathcal{A}: a \equiv 0 \pmod d \},\hspace{15pt} \delta(d)=\prod_{\ell|d }\delta({\ell})\hspace{15pt} {\rm and }\hspace{15pt} 
r_d=| \mathcal{A}_d|-\delta(d)X.$$
Notice that the function $\delta$ depends on both $\mathcal{A}$ and $\mathcal{P}$.
For a real number $z>0$, let
$$P(z)=\prod_{\ell \in \mathcal{P}, \ell <z}\ell \hspace{15pt} {\rm and} \hspace{15pt} W(z)=\prod_{\ell \in \mathcal{P}, \ell <z}(1-\delta(\ell)).$$

\begin{hyp}\label{hyp}
	For the above setup, we now state a series of hypotheses.
	\begin{itemize}
		\item[{\underline{\bf $\Omega_1$:}}]
		There exists a constant $A_1>0$ such that 
		$$0 \le \delta{(\ell)}\le 1-\frac{1}{A_1},~~
		\textit{\rm ~~for all}~ \ell \in \mathcal{P}.$$
		\item[{\underline{\bf $\Omega_2(1,L)$:} }]
		If $2 \le w \le z$, then 
		$$-L \leq \sum_{w \le \ell \le z}\delta(\ell)\log{\ell} -\log{\frac{z}{w}} \le A_2,$$
		where $A_2 \ge 1$ and $L \ge 1$ are some constants independent of $z$ and $w$.\\
		\item[{ \underline{\bf $R(1,\alpha)$:} }]
		There exist $0 <\alpha <1$ and $A_3, A_4\geq 1$ such that for $X\ge 2$ 
		$$\sum_{d \le \frac{X^{\alpha}}{(\log X)^{A_3}}}\mu(d)^2 3^{\omega(d)}|r_d|
		\le A_4\frac{X}{(\log X)^2}.$$
	\end{itemize}
\end{hyp}

For $ \mathcal{A}$ and $\mathcal{P}$ as above and for real numbers $u, v$ and $\lambda$ with $u \le v$, define the weighted sum
\begin{equation}
	\mathcal{W}(\mathcal{A}, \mathcal{P}, v, u, \lambda)=\sum_{\substack{a \in \mathcal{A}\\ (a,P(X^{1/v}))=1}}\Big(1-\sum_{\substack{X^{1/v}\le q< X^{1/u}\\ q|a, q\in \mathcal{P}}}\lambda \Big(1-u\frac{\log q}{\log X}\Big)\Big).
\end{equation}
We now state the following form of Richert’s weighted one-dimensional sieve.

\begin{theorem}[{\cite[Theorem 9.1, Lemma 9.1]{HR}}]\label{HRT}
	With notation as above, assume that the hypotheses $\Omega_1, \Omega_2(1,L)$ and $R(1,\alpha)$ hold for suitable constants $L$ and $\alpha$.
	Suppose further that there exists $u,v, \lambda \in \mathbb{R}$ and $A_5 \ge 1$ such that
	$$
	\frac{1}{\alpha}<u<v, ~~~~ \frac{2}{\alpha}\le v \le \frac{4}{\alpha}, ~~~~
	0<\lambda< A_5.$$
	Then 
	\begin{equation*}
		\mathcal{W}(\mathcal{A}, \mathcal{P}, v, u, \lambda)\ge XW(X^{1/v})
		\Big(F(\alpha,v,u,\lambda)-\frac{cL}{(\log X)^{1/14}}\Big),
	\end{equation*}
	where $c$ is a constant depends at most on $u$ and $v$ (as well as on the $A_i$'s and $\alpha$) and 
	\begin{equation}\label{F}
		F(\alpha,v,u,\lambda)=\frac{2e^{\gamma}}{\alpha v}\Big(
		\log(\alpha v-1)-\lambda \alpha u \log{\frac{v}{u}}+\lambda(\alpha u-1)
		\log{\frac{\alpha v-1}{\alpha u-1}}\Big).
	\end{equation}
	Here $\gamma$ is the Euler's constant and $X$ is the approximation of $\mathscr A$ given in \eqref{approximation}.
\end{theorem}

\section{Proof of \thmref{main}}\label{S_proof_main}
We shall closely follow the arguments of \cite{jos}. The idea is to apply \thmref{HRT} to the following situation.
\begin{align*}
	\mathcal{A}&:=\{|a_1(p)-a_2(p)|: p \le X, a_1(p)\neq a_2(p)\} {\rm ~~~~and~~~~}
	\mathcal{P}:=\{\ell: \ell  \ge M\},
\end{align*}
where $M=M(f_1,f_2)$ is the constant appeared in Section \ref{S_technical}.
It is clear that for any $\ell \in \mathcal{P}$
$$
|\mathcal{A}_\ell|=|\{p\le X: a_1(p)\neq a_2(p), \ell|(a_1(p)-a_2(p))\}|= 
\pi_{f_1,f_2}^*(X,\ell).
$$ 
Applying \propref{asymptotic_pi*}, under GRH,  we obtain
$$| \mathcal{A}_\ell |=\delta({\ell })\frac{X}{\log X}+r_{\ell},$$ 
where $r_\ell=O(\ell^6X^{{1}/{2}}\log (\ell NX))+O(X^{13/14})$.
Furthermore,  if $d$ is a square free integer composed of primes from $\mathcal{P}$, then  from \propref{delta_multiplicative_large_prime} and \propref{asymptotic_pi*} we have
\begin{equation}\label{rd1}
	\delta(d)=\prod_{\ell|d, \ell \in \mathcal{P}}\delta{(\ell)} ~{\rm and}~ r_d=O(d^6X^{{1}/{2}}\log (d NX))+O(X^{13/14}).
\end{equation}
To apply \thmref{HRT}, we now verify that hypotheses $\Omega_1, \Omega_2(L,1)$ and $R(1,\alpha)$, given in Hypothesis \ref{hyp}, hold for our choice of $\mathcal{A}$ and $\mathcal{P}.$ 

\begin{lem}\label{o1}
	Let $f_1$ and $f_2$ be newforms as before. Then we have the following.
	\begin{enumerate}
		\item
		Hypothesis $\Omega_1$ holds with a suitable $A_1$.
		\item
		{Hypothesis $\Omega_2(1,L)$ holds with a suitable $L$.}
		\item
		Under GRH, the  hypothesis $R(1,\alpha)$ holds with any $\alpha < \frac{1}{14}$. 
	\end{enumerate}
\end{lem}
\begin{proof}
	By \propref{asymptotic_delta} the validity of hypotheses $\Omega_1$ and $\Omega_2(1,L)$ are immediate because if $\ell \in P$ then $\delta(\ell)\sim \frac{1}{\ell}$  and this proves hypothesis $\Omega_1$ while the latter one can be achieved by using  Mertens's theorem (cf. \cite[Lemmas 4.6.1, 4.6.2, 4.6.3]{jos}). So we only give a proof of part (3). From  \cite[p. 260]{HW}, we know that 
	$3^{\omega(n)}\le d(n)^{3\log 3/\log 2}\ll n^{\varepsilon}$. Therefore, for any positive constant $A_3$, from \eqref{rd1}, we have
	\begin{align*}
		\sum_{d \le \frac{X^{\alpha}}{(\log X)^{A_3}}}\mu(d)^2 3^{\omega(d)}|r_d|
		&\ll	\sum_{d \le \frac{X^{\alpha}}{(\log X)^{A_3}}} \left(d^{6+\varepsilon}X^{1/2}\log(dNX) + X^{13/14}\right).
	\end{align*}
	We now see that for any  $\alpha< 1/14$
	$$
	\sum_{d \le \frac{X^{\alpha}}{(\log X)^{A_3}}}\mu(d)^2 3^{\omega(d)}|r_d|\ll 
	\frac{X}{(\log X)^2}$$
	and this completes the proof.
\end{proof}
Next we need to choose sieve parameters $\alpha, u,v, \lambda$ satisfying conditions in \thmref{HRT}. For $k\ge 2$ we take:
$$\alpha=\frac{k-1}{14k};~~u=\frac{14k+1}{k-1};~~ v=\frac{56 k}{k-1};~~\lambda=\frac{1}{k^{1/5}}.$$
One can easily verify that these parameters satisfy the conditions required for applying \thmref{HRT} and hence for our choices of $\mathcal{A}$ and $\mathcal{P}$, we obtain
\begin{equation*}
	\mathcal{W}(\mathcal{A}, \mathcal{P}, v, u, \lambda)\gg \frac{X}{(\log X)^2}
	\Big(F(\alpha,v,u,\lambda)-\frac{cL}{(\log X)^{1/14}}\Big).
\end{equation*}
Note that here we have used the fact that $|\mathcal{A}|\gg \frac{X}{\log X}$   and $ W(X)\gg \frac{1}{\log X}$ 
for $X\gg 0$
which follows immediately by using \propref{asymptotic_delta}. Also for the above choices of sieve parameters $\alpha, u,v, \lambda$, the function $F(\alpha,v,u,\lambda)$, defined by \eqref{F},
can be computed explicitly and is given by
$$F\left(\frac{k-1}{14k}, \frac{56k}{k-1},\frac{14k+1}{k-1}, \frac{1}{k^{1/5}}\right)=
\frac{e^{\gamma}\left(14k^{6/5}\log 3+\log 42 k-(1+14k)\log \left(\frac{56k}{14k+1}\right)\right)}{28k^{6/5}}.$$
Moreover,  $F(\alpha,v,u,\lambda)>0$ for $k>1.71\cdots$. Therefore for a fixed weight $k\ge 2$ one can choose $X$, sufficiently large, such that 
$F(\alpha,v,u,\lambda)-\frac{cL}{(\log X)^{1/14}}>0.$  In other words, we have 
\begin{equation}\label{w1}
	\mathcal{W}(\mathcal{A}, \mathcal{P}, v, u, \lambda)\gg \frac{X}{(\log X)^2}.
\end{equation}
Since there are at least $\frac{X}{(\log X)^2}$ many primes $p \le X$ which make a positive contribution to the left hand side of \eqref{w1} therefore to complete the proof of \thmref{main} it is sufficient to show that for any such prime $p$
$$\omega(a_1(p)-a_2(p))\le  [ 7k+{1}/{2}+k^{1/5}  ] .$$
Let $p$ be such a prime. Then $({a_1(p)-a_2(p)},X^{1/v})=1$ and 
\begin{equation}\label{w2}
	1-\sum_{\substack{X^{1/v}\le q< X^{1/u}\\ q|{(a_1(p)-a_2(p))}}}\lambda \Big(1-u\frac{\log q}{\log X}\Big)>0.
\end{equation}
Therefore, we write
\begin{align}\label{idd}
	\omega(a_1(p)-a_2(p))=\sum_{q |{(a_1(p)-a_2(p))}}1=\sum_{\substack{X^{1/v}< q < X^{1/u}\\ q|{(a_1(p)-a_2(p))}}}1+\sum_{\substack{q \ge X^{1/u}\\ q|{(a_1(p)-a_2(p))}}}1.
\end{align}
Now to estimate the first sum on the right of \eqref{idd} we use \eqref{w2} and obtain
\begin{equation*}\label{w3}
	\sum_{\substack{X^{1/v}< q < X^{1/u}\\ q|{(a_1(p)-a_2(p))}}}1<\frac{1}{\lambda}+u\sum_{\substack{X^{1/v}< q < X^{1/u}\\ q|{(a_1(p)-a_2(p))}}}\frac{\log q}{\log X}.
\end{equation*}
For the second sum we observe that if $q \ge X^{1/u}$ then $\frac{\log q}{\log X}\ge \frac{1}{u}$ that gives
\begin{equation*}\label{w4}
	\sum_{\substack{q \ge X^{1/u}\\ q|{(a_1(p)-a_2(p))}}}1\le u\sum_{\substack{q \ge X^{1/u}\\ q|{(a_1(p)-a_2(p))}}}\frac{\log q}{\log X}.
\end{equation*}
Substituting the last two inequalities in \eqref{idd} yields
\begin{align*}
	\omega({a_1(p)-a_2(p)})
	& \le \frac{1}{\lambda}+u\sum_{q|{(a_1(p)-a_2(p))}}\frac{\log q}{\log X} \le\frac{1}{\lambda}+u\frac{\log {(|a_1(p)-a_2(p)|)}}{\log X}.
\end{align*}
Using Deligne's estimate  we know 
$|{a_1(p)-a_2(p)}|\le 4p^{(k-1)/2}$. Therefore for any $p\le X$ as above, we have
\begin{equation*}
	\omega({a_1(p)-a_2(p)})\le \frac{1}{\lambda}+u\frac{k-1}{2}+u\frac{\log 4}{\log X}.
\end{equation*}
Finally substituting the values of $u$ and $\lambda$ and choosing $X$ large enough completes the proof.

\section{Proof of \thmref{Omega}}\label{S_proof_Omega}
The idea of the proof is similar to the proof of \thmref{main} with minor modifications.  We shall apply \thmref{HRT} with the same setting as in Section \ref{S_proof_main}. For $k\ge 2$, we choose the sieve parameters as follows.
$$\alpha=\frac{k-1}{14k};~~u=\frac{26k+1}{k-1};~~ v=\frac{30 k}{k-1};~~\lambda=\frac{1}{\sqrt{\log k}}.$$
Again, these parameters satisfy the conditions required for \thmref{HRT} and
the corresponding function $F(\alpha,v,u,\lambda)>0$ for $k>1.006$.
Hence as in the proof of \thmref{main}, the corresponding weighted sum satisfies
\begin{equation}\label{cw1}
	\mathcal{W}(\mathcal{A}, \mathcal{P}, v, u, \lambda)\gg \frac{X}{(\log X)^2}.
\end{equation}
Next we observe that 
\begin{align*}
	|\{p \le X: \ell^2|(a_1(p)-a_2(p)), ~X^{1/v}\le \ell\le X^{1/u} \}|
	&=\sum_{X^{1/v}\le \ell\le X^{1/u}} \left(\pi_{f_1,f_2}(X,\ell^2) +O(1) \right),
\end{align*}
where the error term is due to the presence of those primes $p$ such that $p|\ell N$ and  $ \ell^2|(a_1(p)-a_2(p))$. Applying \propref{asymptotic_pi} gives that the left side of the above equality is equal to 
$$
{\pi(X)} \sum_{X^{1/v}\le \ell\le X^{1/u}} \frac{1}{\ell^2}+O\left(X^{{1}/{2}+\epsilon}\sum_{X^{1/v}\le \ell\le X^{1/u}} \ell^{12}  \right).
$$
Since $u>26$, we have
\begin{align}\label{ss}
	|\{p \le X: \ell^2|(a_1(p)-a_2(p)), ~X^{1/v}\le \ell\le X^{1/u} \}| 
	=o\Bigg(\frac{X}{(\log X)^2}\Bigg).
\end{align}
We conclude by combining \eqref{cw1} and \eqref{ss} that there are at least $\frac{X}{(\log X)^2}$ many primes $p\le X$ such that \begin{itemize}
	\item [(a)]
	$a_1(p)-a_2(p)$ does not have any prime divisors less than $X^{1/v}$,
	\item[(b)]
	for primes $\ell |(a_1(p)-a_2(p))$ with $X^{1/v}<\ell< X^{1/u}$, 
	${\ell}^2\nmid (a_1(p)-a_2(p)),$
	\item[(c)]
	the contribution of $p$ to the sifting function $\mathcal{W}(\mathcal{A}, \mathcal{P}, v, u, \lambda)$ is positive, i.e., 
	\begin{equation*}\label{cw2}
		1-\sum_{\substack{X^{1/v}\le q< X^{1/u}\\ q \Vert{(a_1(p)-a_2(p))}}}\lambda \Big(1-u\frac{\log q}{\log X}\Big)>0.
	\end{equation*}
\end{itemize}
In order to complete the proof, we will show that if $p\le X$ is a prime satisfying the three conditions above then 
$$\Omega(a_1(p)-a_2(p))\le  [ 13k+{1}/{2}+\sqrt{\log k}  ] .$$
Let $p\le X$ be a prime satisfying  (a), (b) and (c). Then as in the proof of \thmref{main}, we have
\begin{align*}
	\Omega(a_1(p)-a_2(p))
	&= \sum_{\substack{X^{1/v}< q< X^{1/u}\\ q\Vert {(a_1(p)-a_2(p))}}}1+\sum_{\substack{q \ge X^{1/u}\\ q^m|{(a_1(p)-a_2(p))}}}1
	<\frac{1}{\lambda}+u\sum_{q^m|(a_1(p)-a_2(p))}\frac{\log q}{\log X}
\end{align*}
which gives
$$
\Omega(a_1(p)-a_2(p)) \le \frac{1}{\lambda}+u\frac{\log (|a_1(p)-a_2(p)|)}{\log X}.
$$
Now applying Deligne's estimate and arguing as in the proof of \thmref{main}, we get  the desired result.


\begin{acknowledgements} 
	The authors  thank	Prof. Shaunak Deo, Prof. Satadal Ganguly and  Dr. Siddhesh Wagh for many useful discussions and their suggestions on an earlier version of the paper. They would like to express their sincere gratitude to Prof. M. Ram Murty for reading the manuscript, providing  valuable comments and also for sending his paper \cite{mmp}. The authors thank their  respective institutes  for providing excellent working condition and also for financial support.
\end{acknowledgements}

\end{document}